\documentclass{amsart}
\usepackage{amssymb,latexsym,amsthm}
\usepackage[centertags]{amsmath}
\usepackage{amsfonts}
\usepackage{amscd}
\usepackage{graphics}
\usepackage{stmaryrd}

\renewcommand{\Im}{\operatorname{Im}}

\newcommand{\abs}[1]{\left\lvert{#1}\right\rvert}
\newcommand{\norm}[1]{\left\|{#1}\right\|}

\newcommand{\ol}{\overline}
\newcommand{\li}{\hat}

\newcommand{\width}{\operatorname{width}}

\newcommand{\GL}{\mathrm{GL}}

\newcommand{\R}{\mathbb{R}}
\newcommand{\N}{\mathbb{N}}
\newcommand{\Z}{\mathbb{Z}}
\newcommand{\Q}{\mathbb{Q}}
\newcommand{\T}{\mathbb{T}}
\newcommand{\C}{\mathbb{C}}

\DeclareMathOperator{\diam}{diam}

\newcommand{\diff}{\operatorname{Diff}}

\newcommand{\fh}{\mathcal{O}}
\newcommand{\fhc}{\overline{\mathcal{O}}^\infty}

\newcommand{\ie}{i.e.\ }

\newtheorem{theorem}{Theorem}[section]

\newtheorem{corollary}[theorem]{Corollary}
\newtheorem{lemma}[theorem]{Lemma}
\newtheorem{proposition}[theorem]{Proposition}

\newtheorem{claim}{Claim}
\newtheorem*{theorem*}{Theorem}

\theoremstyle{definition}
\newtheorem{definition}[theorem]{Definition}

\theoremstyle{remark}
\newtheorem{remark}[theorem]{Remark}

\begin{document}

\title[A mixing-like property for minimal diffeomorphisms of the 2-torus]{A mixing-like property and inexistence of invariant foliations for minimal diffeomorphisms of the 2-torus}

\author[A. Kocsard]{Alejandro Kocsard}

\address{Universidade Federal Fluminense, Instituto de Matem\'atica, Rua M\'ario Santos Braga S/N, 24020-140 Niteroi, RJ, Brasil}

\email{alejo@impa.br}

\author[A.  Koropecki]{Andr\'es Koropecki}

\address{Universidade Federal Fluminense, Instituto de Matem\'atica, Rua M\'ario Santos Braga S/N, 24020-140 Niteroi, RJ, Brasil}

\email{koro@mat.uff.br}
\thanks{The authors were supported by CNPq-Brasil.}

\date{}

\begin{abstract} 
We consider diffeomorphisms in $\fhc(\T^2)$, the $C^\infty$-closure of the conjugancy class of translations of $\T^2$. By a theorem of Fathi and Herman, a generic diffeomorphism in that space is minimal and uniquely ergodic. We define a new mixing-type property, which takes into account the ``directions'' of mixing, and we prove that generic elements of $\fhc(\T^2)$ satisfy this property. As a consequence, we obtain a residual set of strictly ergodic diffeomorphisms without invariant foliations of any kind. We also obtain an analytic version of these results.
\end{abstract}
\maketitle

\section{Introduction} 

In \cite{fathi-herman}, Fathi and Herman combined generic arguments with the so-called \emph{fast approximation by conjugations} method of Anosov and Katok \cite{anosov-katok}, to study a particular class of diffeomorphisms of a compact manifold: the $C^\infty$-closure of the set of diffeomorphisms which are $C^\infty$ conjugate to elements of a locally free $\T^1$-action on the manifold. They proved that a generic element of that space is minimal and uniquely ergodic (\ie there is a residual subset of such diffeomorphisms), in particular proving that every compact manifold admitting a locally free $\T^1$-action supports a minimal and uniquely ergodic diffeomorphism. 

Surprisingly, the space studied by Fathi and Herman contains many elements with unexpected dynamical properties; for example, a generic diffeomorphism in that space is weak mixing \cite{herman:lagrangian, fayad-sap}, and if the underlying space is $\T^n$, the action of its derivative on the unit tangent bundle is minimal \cite{tesis}. For a very complete survey on the technique of Anosov-Katok and its variations, see \cite{fayad-katok}. In \cite{fayad-sap}, Fayad and Saprikyna use a real analytic version of this method to construct minimal weak mixing diffeomorphisms.

In this short article, we restrict our attention to diffeomorphisms of $\T^2$. In this setting, the closure of maps $C^\infty$ conjugated to elements of any locally free $\T^1$-action coincides with the $C^\infty$-closure of the conjugancy class of the rigid translations $$R_{(\lambda_1,\lambda_2)} \colon (x,y)\mapsto (x+\lambda_1,y+\lambda_2),$$ 
that is, the set $\fhc(\T^2)$, where
$$\fh(\T^2) = \left\{hR_\alpha h^{-1}:h\in \diff^\infty(\T^2),\,\alpha\in\T^2\right\}.$$

As we mentioned above, a generic element of $\fhc(\T^2)$ is topologically weak mixing; however, no topologically mixing elements are known. It is also unknown if a minimal diffeomorphism of $\T^2$ can be topologically mixing. In fact, no examples of minimal $C^\infty$ diffeomorphisms of $\T^2$ in the homotopy class of the identity are known other than the ones in $\fhc(\T^2)$.

Recall that a homeomorphism $f\colon \T^2\to \T^2$ is topologically weak mixing if $f\times f$ is transitive. An equivalent definition is the following: for each open $U\subset \T^2$ and $\epsilon>0$, there is $n>0$ such that $f^n(U)$ is $\epsilon$-dense in $\T^2$. We will be interested in a similar property, which implies weak-mixing but is stronger in that it requires open sets to be mixed in every homological direction.
\begin{definition} A homeomorphism $f\colon \T^2\to \T^2$ is \emph{weak spreading} if for a lift $\li{f}\colon \R^2\to \R^2$ of $f$ the following holds: for each open set $U\in \R^2$, $\epsilon>0$ and $N>0$, there is $n>0$ such that $\li{f}^n(U)$ is $\epsilon$-dense in a ball of radius $N$. 
\end{definition}

Let $\fhc_\mu(\T^2)$ denote the area-preserving elements of $\fhc(\T^2)$. Now we can state our main theorem.

\begin{theorem} \label{theorem:main} Weak spreading diffeomorphisms are generic in $\fhc(\T^2)$ and $\fhc_\mu(\T^2)$.
\end{theorem}

As a consequence, we prove a result about invariant foliations announced by Herman in \cite{fathi-herman} without proof.  By a topological foliation we mean a codimension-$1$ foliation of class $C^0$; that is, a partition $\mathcal{F}$ of $\T^2$ into one-dimensional topological sub-manifolds which is locally homeomorphic to the partition of the unit square by horizontal segments. We say that the foliation is invariant by $f$ if $f(F)\in \mathcal{F}$ for every $F\in \mathcal{F}$. We then have:

\begin{corollary}\label{corollary:main}
The set of diffeomorphisms in $\fhc(\T^2)$ (resp. $\fhc_\mu(\T^2)$) without any invariant topological foliation is residual in $\fhc(\T^2)$ (resp. $\fhc_\mu(\T^2)$). 
\end{corollary}

Since the set of minimal and uniquely ergodic diffeomorphisms in $\fhc(\T^2)$ (or $\fhc_\mu(\T^2)$) is also residual, this provides a residual set of minimal, uniquely ergodic diffeomorphisms with no invariant foliations.

Using the ideas of \cite{fayad-sap}, it is possible to construct real analytic examples, working with diffeomorphisms which have an analytic extension to a band of fixed width in $\C^2$ (see the precise definitions in \S\ref{sec:analytic}).

\begin{theorem}\label{theorem:anal} The set of real analytic diffeomorphisms of $\T^2$ which are weak spreading is residual in $\ol{\mathcal{O}}^\omega_\rho(\T^2)$.
\end{theorem}

\begin{remark}
We use the word ``weak'' in the definition of weak spreading because there is an analogy with the topological weak mixing property. We could also define \emph{strong} spreading (or just \emph{spreading}) as the property that for any open set $U\subset \R^2$, $\epsilon>0$ and $N>0$ there is $n_0$ such that that $\li{f}^n(U)$ is $\epsilon$-dense in a ball of radius $N$ whenever $n>n_0$. This would be in analogy with the definition of topological mixing, but it is clearly a stronger property. In fact, the typical examples of topologically mixing systems in $\T^2$ mix only in one direction (e.g. Anosov systems and time-one maps of some minimal flows \cite{fayad}). It is not obvious that strong spreading diffeomorphisms exist; however, as P. Boyland kindly explained to us, an example of a strong spreading diffeomorphism can be constructed using Markov partitions and the techniques of \cite{boyland}.
\end{remark}

\subsection{Acknowledgments}
We are grateful to E. Pujals and P. Boyland for useful discussions, and the anonymous referee for bringing the results of \cite{fayad-sap} to our attention and suggesting various improvements, in particular the content of $\S\ref{sec:analytic}$.

\section{The method of Fathi-Herman}

As usual, we identify $\T^2\simeq \R^2/\Z^2$ with quotient projection $\pi\colon \R^2\to \T^2$, and denote by $\diff^\infty(\T^2)$ the space of $C^\infty$ diffeomorphisms of $\T^2$. A lift of one such diffeomorphism $f$ to $\R^2$ is a map $\li{f}\colon \R^2\to \R^2$ such that $f\pi = \pi\li{f}$. If $f$ is homotopic to the identity, this is equivalent to saying that $\li{f}$ commutes with integer translations, \ie $\li{f}(z+v) = \li{f}(z)+v$ for $v\in \Z^2$. Two different lifts of a diffeomorphism of $\T^2$ always differ by a constant $v\in \Z^2$. We will denote by $\li{R}_\alpha$ the translation $z\mapsto z+\alpha$ of $\R^2$ and by $R_\alpha$ the rotation of $\T^2$ lifted by $\li{R}_\alpha$.

The method used in \cite{fathi-herman}, adapted to our case, can be resumed as follows:

\begin{lemma}\label{lem:fh}
Let $P\subset \fhc(\T^2)$ (or $\fhc_\mu(\T^2)$) be such that
\begin{enumerate}
	\item $P=\bigcap_{n\geq 0} P_n$, where the $P_n$ are open; 
	\item For each $g\in \diff^\infty(\T^2)$ (resp. $\diff^\infty_\mu(\T^2)$) and $m\in \N$, there is $N>0$ such that $\{gfg^{-1} : f\in P_n\} \subset P_m$ whenever $n>N$;
	\item For each $n\in \N$, $p/q\in \Q$, there exists $h\in\diff^\infty(\T^2)$ (resp. $\diff^\infty_\mu(\T^2)$) such that
		\begin{itemize} 
			\item $hR_{(1/q,0)} = R_{(1/q,0)}h$;
			\item $hR_{\alpha_k} h^{-1}\in P_n$ for some sequence $\alpha_k\to (p/q,0)$.
		\end{itemize}
\end{enumerate}
Then, $P$ is residual in $\fhc(\T^2)$ (resp. $\fhc_\mu(\T^2)$). 
\end{lemma}

\begin{proof}
Given $m\in \N$, $p/q\in \Q$, and $g\in \diff^\infty(\T^2)$, let $n$ be as in (2), and then $h$ and $\{\alpha_k\}$ as in (3). Then 
	$$P_n \ni hR_{\alpha_k}h^{-1}\xrightarrow[k\to\infty]{C^\infty} hR_{(p/q,0)}h^{-1}=R_{(p/q,0)},$$
so that
$$P_m\ni g(hR_{\alpha_k}h^{-1})g^{-1}\xrightarrow[k\to\infty]{C^\infty}  gR_{(p/q,0)}g^{-1}.$$
This proves that $gR_{(p/q,0)}g^{-1}\in \ol{P}_m^\infty$. Since this holds for all $g$ and $p/q$, it follows that $P_m$ is dense in $\fhc(\T^2)$ because so is the set
			$$\left\{hR_{(p/q,0)}h^{-1} : h\in \diff^\infty(\T^2),\, p/q\in \Q\right\}.$$
Since this holds for all $m$ and each $P_m$ is open, this proves that $P$ is residual in $\fhc(\T^2)$. The proof in the area-preserving case is the same.
\end{proof}

The property of having no invariant topological foliations is hard to deal with in the $C^\infty$ topology in order to apply the above lemma. However, the weak spreading property can be adequately described as an intersection of countably many properties that fit well into the lemma; thus we first prove Theorem \ref{theorem:main} using the above method, and then we prove that weak spreading is not compatible with the existence of invariant foliations of any kind, which implies  Corollary \ref{corollary:main}.

\section{Proof of Theorem \ref{theorem:main}}

Let $P_n$ denote the set of all $f\in \fhc(\T^2)$ such that if $\li{f}$ is a lift of $f$, for any ball $B$ of radius $1/n$ in $\R^2$, there is $k>0$ such that $\li{f}^k(B)$ is $1/n$-dense in a ball of radius $n$. Note that if this property holds for some lift, it holds for any lift of $f$. 

It is clear that $P=\cap P_n$ is the set of weak spreading elements of $\fhc(\T^2)$. 
Denote by $B(z,\epsilon)$ the ball of radius $\epsilon$ centered at $z$. Given a lift $\li{f}$ of some $f\in P_n$, and $z\in \R^2$, let $k_z$ be the smallest positive integer such that $\li{f}^{k_z}(B(z,1/n))$ is $1/n$-dense in a ball of radius $n$. By continuity of $\li{f}$, the map $z\mapsto k_z$ is upper semi-continuous, and therefore it attains a maximum $K$ when $z\in [0,1]^2$. But since $\li{f}$ lifts a map homotopic to the identity, $k_z = k_{z+v}$ when $v\in \Z^2$, so that $k_z\leq K$ for all $z\in \R^2$. Hence, if $g$ is close enough to $f$ in the $C^0$ topology and $\li{g}$ is the lift of $g$ closest to $\hat{f}$, it also holds that $\li{g}^{k_z}(B(z,1/n))$ is dense in a ball of radius $n$ for any $z\in \R^2$. Hence $P_n$ is open in the $C^0$ topology (and, in particular, in the $C^\infty$ topology).

To see that condition (2) of Lemma \ref{lem:fh} holds, note that any lift $\li{g}$ of a diffeomorphism $g$ of $\T^2$ is bi-Lipschitz. Fix $m\in \N$, let $C$ be a Lipschitz constant for $\li{g}$ and $\li{g}^{-1}$, and let $n>0$ be such that $C<n/m$. If $\li{f}$ is a lift of $f\in P_n$, and if $U$ is an open set, then there is $k$ such that $\li{f}^k(\li{g}^{-1}(U))$ is $1/n$-dense in a ball of radius $n$. Thus, $\li{g}\li{f}^k\li{g}^{-1}(U)$ is $C/n$-dense in a ball of radius $n/C$, which implies that $gfg^{-1}\in P_m$ as required.

To finish the proof, it remains to see that condition (3) of  Lemma \ref{lem:fh} holds. To do this, it suffices to construct, for each $q,n\in \N$, a diffeomorphism $h\in \diff^\infty(\T^2)$ which commutes with $R_{(1/q,0)}$ and such that $hR_{(\alpha,0)}h^{-1}\in P_n$ whenever $\alpha$ is irrational. Note that it is enough to prove this for some multiple of $q$ instead of $q$. We will assume that $q$ is a multiple of $n$ and $q\geq 2n$, since otherwise we may use $2qn$ instead of $q$. We define $h$ by constructing a lift $\li{h}=\li{v}\circ\li{u}$, where  $\li{v}, \li{u}\colon \R^2\to\R^2$ are the maps
$$\li{u}(x,y)=\left(x,\, y + m\cos(2\pi q x)\right),\quad \li{v}(x,y)=\left(x+n\cos(2\pi qy),y\right)$$ 
and $m$ is a sufficiently large integer that we will choose later.
It is clear that $\li{u}$ and $\li{v}$ are lifts of $C^\infty$ torus diffeomorphisms in the homotopy class of the identity, because they commute with integer translations. They also commute with $\li{R}_{(1/q,0)}$ and $\li{R}_{(0,1/q)}$ as well. The same properties hold for $\li{h}=\li{v}\circ\li{u}$. Moreover, since both $\li{u}$ and $\li{v}$ are area-preserving, so is $\li{h}$ and the rest of the proof also works in the area-preserving setting.

\begin{figure}[t]
  \centering 
   \resizebox{7cm}{!}{\includegraphics{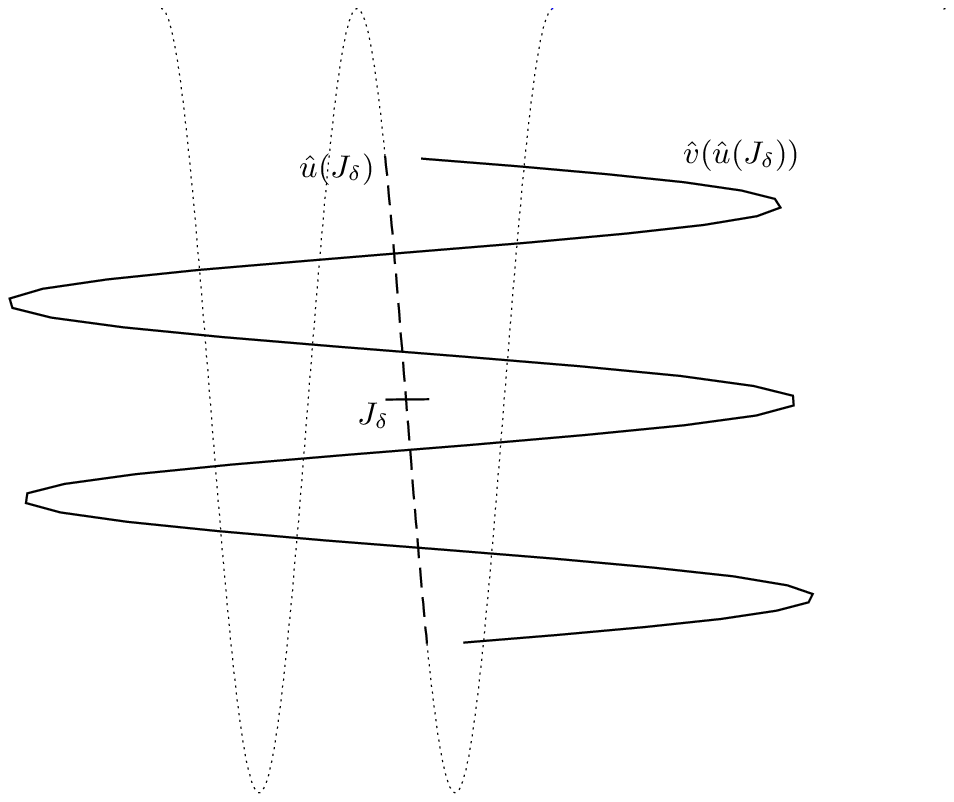}}\\ 
  \caption{Image of $J_\delta$ by $\li{h}$}
  \label{fig1}
\end{figure}

Let $\delta = 2n(\pi q m)^{-1}$, $I_\delta = [-\delta, \delta]\times \{0\}$, and $J_\delta = [(4q)^{-1}-\delta/2, (4q)^{-1}+\delta/2]\times \{0\}$. 
\begin{claim} If $m$ is large enough, then $\li{h}(I_\delta)$ is contained in the ball of radius $1/(2n)$ centered at $(n,m)$, and $\li{h}(J_\delta)$ is $1/n$-dense in $[-n,n]\times[-n,n]$.
\end{claim}

\begin{proof}
First observe that from the inequality
$$1-\cos(x) \leq x^2/2 \quad \forall x\in \R$$
it follows that (denoting by $(x_1,x_2)_i$ the coordinate $x_i$)
$$\abs{(\li{u}(x,0) - \li{u}(0,0))_2} \leq 2m(\pi q x)^2 < 2m(\pi q \delta)^2 = 8n^2/m$$
if $\abs{x}<\delta$. Since $\li{u}(0,0)=(0,m)$, this means that $\li{u}(I_\delta)$ is contained in the rectangle $[-\delta,\delta]\times [m-b, m+b]$ where $b=8n^2/m$. By the definition of $\li{v}$ and a similar argument (since $m$ is an integer), we can conclude that $\li{v}(\li{u}(I_\delta))\subset [n-a,n+a]\times[m-b,m+b]$, where
$$a=\delta + 2n(\pi qb)^2 = 2n(\pi q m)^{-1} + 128\pi^2q^2n^5m^{-2}.$$
Since both $a$ and $b$ can be made arbitrarily small if $m$ is large enough, $\li{h}(I_\delta)$ is contained in a ball around $(n,m)$ of radius $1/(2n)$ if $m$ 
is large enough.

For the second part of the claim, note that
$$\cos(x+\pi/2) = \sin(x) \geq x/2 \quad \text{ if }0\leq x\leq \pi/2$$
so that 
$$\left(\li{u}((4q)^{-1}+\delta/2,0)-\li{u}((4q)^{-1},0)\right)_2= m\cos(\pi q\delta + \pi/2) \geq m\pi q\delta/2 = n,$$
and similarly
$$\left(\li{u}((4q)^{-1}-\delta/2,0)-\li{u}((4q)^{-1},0)\right)_2 = -n.$$
Thus $\li{u}(J_\delta)$ is an arc that transverses vertically the rectangle $[-\delta/2,\delta/2]\times [-n,n]$. 

Let $L= \{0\}\times [-n,n]$. Note that $\li{v}(L)$ is $1/q$-dense in $[-n,n]\times [-n,n]$, since every rectangle of the form $[-n,n]\times [-n+k/q, -n+(k+1)/q]$, $0\leq k \leq 2qn-1$ is horizontally transversed by $\li{v}(L)$. By the previous paragraph, $\li{u}(J_\delta)$ contains a point of the form $(s, y)$ with $\abs{s}<\delta/2$ for each $(0,y)\in L$. Since $\li{v}(s,y) =  \li{v}(0,y)+(s,0)$, it follows from the previous facts that, if $m$ is so large that $\delta/2<1/q$, $\li{h}(J_\delta)= \li{v}(\li{u}(J_\delta))$ is $2/q$-dense in $[-n,n]\times [-n,n]$ (see Figure 1). Since we assumed earlier that $q\geq 2n$, we conclude that $h(J_\delta)$ is $1/n$-dense in $[-n,n]\times [-n,n]$ as claimed. This proves the claim.
\end{proof}

Let $B\subset \R^2$ be a ball of radius $1/n$. Then $B$ contains a ball $B'$ of radius $1/(2n)$ around some point of coordinates $(i/q, j/q)$, with $i,j$ integers (because $q\geq 2n$). Since $\li{h}$ commutes with $R_{(1/q,0)}$, and using Claim 1, we see that $$\li{h}(I_\delta +(i/q,j/q)-(n,m)) = \li{h}(I_\delta) -(n,m) + (i/q,j/q) \subset B'$$
In particular, $I_\delta +(i/q,j/q)-(n,m) \subset \li{h}^{-1}(B)$. Since $J_\delta$ lies on the same horizontal line as $I_\delta$ and is shorter than $I_\delta$, if $\alpha$ is an irrational number we can find $k\in \N$ and $r\in \Z$ such that 
$J_\delta+(r,0)\subset \li{R}^k_{(\alpha,0)}(I_\delta)$, and we have
$$J_\delta+(i/q,j/q)-(n,m)+(r,0) \subset \li{R}^k_{(\alpha,0)}(I_\delta+(i/q,j/q)-(n,m)).$$ 
Thus, if $\li{f} = \li{h}\li{R}_{(\alpha,0)}\li{h}^{-1}$,
\begin{multline*}
\li{f}^k(B) = \li{h}\li{R}^k_{(\alpha,0)}\li{h}^{-1}(B) \supset \li{h}\li{R}^k_{(\alpha,0)}(I_\delta +(i/q,j/q)-(n,m))\\ \supset \li{h}(J_\delta+(i/q,j/q)-(n+r,m)) = \li{h}(J_\delta)+(i/q,j/q)-(n+r,m)
\end{multline*}
which is just a translation of $\li{h}(J_\delta)$, and thus by Claim 1 it is $1/n$-dense in some ball of radius $n$. That is, $\li{f}^k(B)$ is $1/n$-dense in some ball of radius $n$, which means that $hR_{(\alpha,0)}h^{-1}\in P_n$. Since $\alpha$ was an arbitrary irrational number, this completes the proof.
\qed

\section{Invariant foliations}

Corollary \ref{corollary:main} is a direct consequence of Theorem \ref{theorem:main} and the next two propositions.

\begin{proposition} \label{pro:folia} If $\mathcal{F}$ is a foliation of $\T^2$ and $\li{\mathcal{F}}$ is the lift of $\mathcal{F}$ to $\R^2$, then there is a leaf $F\in \li{\mathcal{F}}$ which is contained in a strip bounded by two parallel straight lines $L$ and $L'$, such that both lines belong to different connected components of  $\R^2-F$.
\end{proposition}

\begin{proof}
If $\mathcal{F}$ has a compact leaf, there is $z\in \R^2$ and a leaf $F$ of $\li{\mathcal{F}}$ such that $F+(p,q)=F$, for some pair of integers $(p,q)\neq (0,0)$. Thus, assuming $p\neq 0$, if $L_0$ is a line of slope $q/p$, it holds that $s = \sup\{d(z,L_0): z\in F\}<\infty$, and the proposition follows by choosing $L$ and $L'$ a distance greater than $s$ apart from $L_0$, one on each side. If $p=0$, then $q\neq 0$ an analogous argument holds.

Now suppose $\mathcal{F}$ has no compact leaves. By \cite[Theorem 4.3.3]{hector-hirsch}, $\mathcal{F}$ is equivalent to a foliation $\mathcal{F}'$ obtained by suspension of the trivial foliation $\R\times \T^1$ over an orientation preserving circle homeomorphism $f\colon \T^1\to \T^1$ with irrational rotation number. Such a foliation has a lift $\li{\mathcal{F}}'$ to $\R^2$ such that the intersection of the leaf through $(0,y)$ with the line $\{n\}\times \R$ is at $(n,\li{f}^n(y))$, where $\li{f}\colon \R\to \R$ is a lift of $f$. If $\phi(y)$ denotes the length of the arc of leaf joining $(0,y)$ to $(1,\li{f}(y))$, then $\phi\colon \R \to \R$ is a continuous function and it is $\Z$-periodic, because $\li{\mathcal{F}}'$ is a lift of a foliation of $\T^2$. Thus there is a constant $C$ such that $\phi(x)<C$ for all $x\in \R$. Note that the length of the arc joining $(n,y)$ to $(n+1,\li{f}(y))$ is also bounded by $C$. 

If $\rho$ is the rotation number of $\li{f}$, by classic results for circle homeomorphisms (see, for example, \cite{wellington}) we have $|\li{f}^n(y)-y - n\rho| \leq 1$ for all $n\in \Z$ and $y\in \R$. 
Let $F'$ be a leaf of $\li{\mathcal{F}}'$ containing the point $(0,y)$. Then $F'= \cup_{n\in \Z} F_n'$ where $F_n$ is the arc joining $(n,\li{f}^n(y))$ to $(n+1, \li{f}^{n+1}(y))$. Note that the distance from $(n,\li{f}^n(y))$ to the line $L_0$ of slope $\rho$ through $(0,y)$ is at most $1$, and the length of $F_n'$ is at most $C$. Thus the distance from any point of $F'$ to $L_0$ is at most $C+1$. 

We know that $\mathcal{F}$ is equivalent to $\mathcal{F}'$, which means there is a homeomorphism $h\colon \T^2\to \T^2$ mapping leaves of $\mathcal{F}'$ to leaves of $\mathcal{F}$. If $\li{h}\colon \R^2\to \R^2$ is a lift of $h$, then we can write $\li{h}(z)= A(z) + \psi(z)$ where $A\in \GL(2,\Z)$ and $\psi$ is a $\Z^2$-periodic function, bounded by some constant $K$. If $L_1 = AL_0$, $z = \li{h}(z')$ is a point in $F=h(F')$, and $w=A(w')$ is a point in $L_1$ then 
$$\abs{z-w} = \abs{A(z'-w')+\psi(z)}\leq \norm{A}\cdot\abs{z'-w'} + K \leq \norm{A}(C+1)+K,$$
the last inequality following from the fact that $z'\in F'$ and $w'\in L_0$.
It follows that $s = \sup_{z\in F} d(z,L_1)<\infty$, and as before we complete the proof choosing $L$ and $L'$ parallel to $L_1$ and a distance at least $s$ apart from $L_1$, one on each side.
\end{proof}

\begin{proposition}\label{pro:foliadev} If $f$ is weak spreading and homotopic to the identity, then $f$ has no invariant topological foliations.
\end{proposition}
\begin{proof}
By Proposition \ref{pro:folia}, if $\mathcal{F}$ is a foliation invariant by $f$ and $\li{\mathcal{F}}$ is the lift of this foliation to $\R^2$ (hence invariant by $\li{f}$), there is a leaf $\li{F}_0\in \li{\mathcal{F}}$ which is contained in a strip bounded by two parallel lines $L$ and $L'$, and which contains each of those lines in a different component of its complement.
Let $u$ be a unit vector orthogonal to $L$. We will assume without loss of generality that $u$ has a nonzero second coordinate. If $S$ is the strip bounded by $\li{F}_0$ and $\li{F}_0+(0,1)$, denoting by $\phi_u\colon\R^2\to\R$ the orthogonal projection onto the direction of $u$, it is clear that $$\width_u(S)\doteq \diam(\phi_u(S))<\infty.$$
Moreover, $\cup_{n\in\Z} S+(0,n) = \R^2$.

For each $n\in \Z$, since $\li{f}^n(\li{F}_0)$ cannot cross $\li{F}_0+(0,k)$ for any $k\in \Z$, we see that $\li{f}^n(\li{F}_0) \subset S+(0,m)$ for some $m\in \Z$. This implies that $\width_u({f}^n(\li{F}_0))\leq M = \width_u(S).$ But then 
$$\li{f}^n(\li{F}_0+(0,1)) \subset S+(0,m+1),$$
so that $\li{f}^n(S)$ is contained in the strip bounded by $\li{F}_0+(0,m)$ and $\li{F}_0+(0,m+2)$. This means that $\width_u(\li{f}^n(S))\leq 2M$. However, if $f$ is weak spreading, then there is $n>0$ such that $\li{f}^n(S)$ is $1/3$-dense in some ball of radius $3M$, so that $\width_u(\li{f}^n(S))> 2M$, contradicting the previous claim. This completes the proof.
\end{proof}

\section{The real analytic case}
\label{sec:analytic}

In this section we briefly explain how to obtain minimal weak spreading analytic diffeomorphisms of $\T^2$. We kindly thank the anonymous referee for bringing this to our attention.

First we introduce some notation, following \cite{fayad-sap}. Fix $\rho>0$, and let $g:\R^2\to \R^2$ be any real analytic $\Z^2$-periodic function which can be holomorphically extended to $A_\rho=\{(z,w)\in \C^2 : \abs{\Im{z}}<\rho,\, \abs{\Im{w}}<\rho\}$. We define $\norm{g}_\rho=\sup_{A_\rho}\abs{g(z,w)}$, and we denote by $C^\omega_\rho(\T^2)$ the space of all functions of this kind which satisfy $\norm{g}_\rho<\infty$.

Let $\diff_\rho^\omega(\T^2)$ be the space of all diffeomorphisms $f$ of $\T^2$ which are homotopic to the identity, and which have a lift whose periodic part is in $C_\rho^\omega(\T^2)$. There is a metric in $\diff_\rho^\omega(\T^2)$ defined by
$$d_\rho(h,k)=\inf_{(p,q)\in \Z^2} \norm{\li{h}-\li{k}+(p,q)}_\rho,$$
where $\li{h}$ and $\li{k}$ are lifts of $h$ and $k$, respectively.
Since $C_\rho^\omega(\T^2)$ is a Banach space, it is easy to see that the metric $d_\rho$ turns $\diff^\omega_\rho(\T^2)$ into a complete metric space. 

To apply the arguments of the previous sections we work in the space
$\ol{\mathcal{O}}_\rho^\omega(\T^2)$ defined as the closure in the $d_\rho$ metric of the set of diffeomorphisms of the form $hR_\alpha h^{-1}$ where $\alpha\in \T^1$, and $h\in \diff_\rho^\omega(\T^2)$ is any diffeomorphism whose lifts to $\R^2$ have a bi-holomorphic extension to $\C^2$.

We observe that the proof of Lemma \ref{lem:fh} applies to this setting if we use $\ol{\mathcal{O}}_\rho^\omega(\T^2)$ instead of $\fhc(\T^2)$ (and the topology induced by $d_\rho$ instead of the $C^\infty$ topology). 

To complete the proof of Theorem \ref{theorem:anal}, we note that everything in \S3 works without any modifications, because the function $h$ constructed to obtain property (3) of Lemma \ref{lem:fh} has an analytic extension to all of $\C^2$ which is a bi-holomorphism.

\bibliographystyle{amsalpha} 
\bibliography{tesis}

\end{document}